\documentclass{amsart}
\usepackage{amsmath,amssymb,amscd,mathrsfs,amscd}
\usepackage{graphics,verbatim}
\usepackage{tikz}

\usepackage{hyperref}
\usepackage{bbm}
\usepackage{graphicx}
\usepackage{tikz-cd}

\newcommand{\la}{\langle}
\newcommand{\ra}{\rangle}
\newcommand{\bbk}{\mathbbm{k}}

\DeclareMathOperator{\Hom}{Hom}
\DeclareMathOperator{\Ext}{Ext}
\DeclareMathOperator{\End}{End}
\DeclareMathOperator{\uHom}{\underline{Hom}}

\DeclareMathOperator{\Lie}{Lie}
\DeclareMathOperator{\co}{H}
\newcommand{\ind}[2]{\mathrm{ind}^{#1}_{#2}}
\newcommand{\res}[2]{\mathrm{res}^{#1}_{#2}}

\newcommand{\wts}{\mathbb{X}}
\newcommand{\dom}{\wts^+}
\newcommand{\dm}{\mathrm{dom}} 
\newcommand{\conv}{\mathrm{conv}}

\newcommand{\wtles}{\preceq} 
\newcommand{\wtless}{\prec} 

\newcommand{\gm}{\mathbb{G}_{\mathrm{m}}}

\newcommand{\spr}{\widetilde{\mathcal{N}}}
\newcommand{\nilp}{\mathcal{N}}
\newcommand{\upmu}[1]{\tilde{\mu}_{#1}}

\newcommand{\gr}{\mathcal{G}r}

\newcommand{\str}{\mathcal{O}} 
\newcommand{\vb}{\mathscr{S}} 

\newcommand{\exstd}{\Delta}
\newcommand{\excos}{\nabla}
\newcommand{\expstd}{\bar{\Delta}}
\newcommand{\expcos}{\bar{\nabla}}
\newcommand{\exirr}{\mathscr{L}}

\newcommand{\vd}{\mathbb{D}}

\newcommand{\coh}{\mathrm{Coh}}
\newcommand{\der}{\mathrm{D}^\mathrm{b}}
\newcommand{\dmix}{\mathrm{D}^{\mathrm{mix}}}
\newcommand{\rep}{\mathrm{Rep}}

\newcommand{\tilt}{\mathrm{Tilt}}

\makeatletter
\newcommand{\leqnomode}{\tagsleft@true}
\newcommand{\reqnomode}{\tagsleft@false}
\makeatother

\newtheorem{thm}{Theorem}[section]

\newtheorem{lemma}[thm]{Lemma}

\newtheorem{prop}[thm]{Proposition}

\newtheorem{cor}[thm]{Corollary}

\theoremstyle{definition}
\newtheorem{defn}[thm]{Definition}

\theoremstyle{remark}
\newtheorem{remark}[thm]{Remark}

\numberwithin{equation}{section}
\numberwithin{figure}{section}

\begin{document}
\title{Exotic t-structure for partial resolutions of the nilpotent cone}

\author[K.~Chan]{Kei Yuen Chan}
\address{Shanghai Center for Mathematical Sciences, Jiangwan Campus, Fudan University, Shanghai, China.}
\email{kychan@fudan.edu.cn}
 \author[L.~Rider]{Laura Rider}
\address{Department of Mathematics, University of Georgia, Athens, Georgia 30602, USA.}
\email{laurajoy@uga.edu}
\author[P.~Sobaje]{Paul Sobaje}
\address{Department of Mathematics, Georgia Southern University, Statesboro, Georgia 30458, USA.}
\email{psobaje@georgiasouthern.edu}
\date{\today}

 \thanks{L.R. was supported by NSF Grant No. DMS-1802378.}

\begin{abstract}
We define and study an exotic $t$-structure on the bounded derived category of equivariant coherent sheaves on partial resolutions of the nilpotent cone.  
\end{abstract}

\maketitle
\section{Introduction}

Let $G$ be a connected, reductive algebraic group defined over an algebraically closed field $\Bbbk$ satisfying a technical assumption (see Subsection \ref{subsec:notation}). Let $P$ denote any parabolic subgroup of $G$ and $\nilp_P$ denote the $P\times \gm$-variety of nilpotent elements in $\mathrm{Lie}(P)$. From this data, we can define the following $G\times \gm$-variety
\[\spr^P = G\times^P\nilp_P \cong \{(x, Q) | Q\subset G \textup{ is conjugate to } P \textup{ and } x\in\nilp_Q \}.\] In this note, we study $\der\coh^{G \times \mathbb{G}_m}(\spr^P)$, the bounded derived category of $G\times\gm$-equivariant coherent sheaves on $\spr^P$. This variety plays roles in both the Springer correspondence \cite{BM} and in the generalized Springer correspondence \cite{Lu}. 
\subsection{Extreme cases} When $P = G$, $\spr^G =:\nilp $ is a singular variety more commonly called the nilpotent cone. The geometry of the nilpotent cone and the structure of its singularities have intricate and often surprising connections to the representation theory of $G$ and its Lie algebra. In this case, $\der\coh^{G \times \mathbb{G}_m}(\nilp)$ is known to admit a \textit{perverse coherent} t-structure. This t-structure (defined similarly to the perverse t-structure for constructible sheaves) plays a role in the proofs of the Lusztig--Vogan bijection \cite{BEZ2}, Humphrey's conjecture for quantum groups \cite{BEZ} and for algebraic groups \cite{ahr}, and Mirkovi\'c--Vilonen conjecture \cite{ARd2, MR2} to name a few. 

The key properties which make this non-standard t-structure extremely useful are that
\begin{enumerate}
\item all perverse coherent sheaves have finite length,
\item characterization of the simple perverse coherent sheaves, and finally 
\item the category satisfies recollement or gluing. 
\end{enumerate}
The first two properties follow in a straightforward way from the definition of the t-structure (similar to the case of perverse constructible sheaves). The third property \textit{does not}. In general, sheaf functors that preserve bounded complexes of coherent sheaves need not have adjoints with the same property (in stark contrast with constructible sheaves). Property (3) is a consequence of an alternative construction of this t-structure due to Bezrukavnikov \cite{BEZ2} using a \textit{quasi-exceptional collection} (see  \cite[2.2]{BEZ2} for a definition). 

When $P = B$ is a Borel subgroup of $G$, $\spr^B =: \spr$ identifies with the cotangent bundle of the flag variety associated to $G$, and as such is smooth. Moreover, the natural morphism $\mu: \spr\rightarrow\nilp$ (called the Springer resolution) resolves the singularities of $\nilp$. In this case, there is no perverse coherent t-structure on $\der\coh^{G \times \mathbb{G}_m}(\spr)$. However, there is a non-standard t-structure called the \textit{exotic} t-structure (\cite{BEZ}) which satisfies versions of properties (1)-(3) above, and such that $\mu_*$ takes exotic sheaves to perverse coherent sheaves. This t-structure is defined using an \textit{exceptional collection}. This exotic t-structure also appears repeatedly in representation theoretic applications (\cite{abg, ARd, MR2, AR, amrw}). 

\subsection{Partial resolution} When $B\subset P\subset G$, the Springer resolution factors through $\spr^P$,  $\spr\stackrel{\upmu{P}}{\rightarrow}\spr^P\stackrel{\mu_{P}}{\rightarrow}\nilp$, and $\spr^P$ partially resolves the singularities of $\nilp$, so we refer to $\spr^P$ as a partial resolution of the nilpotent cone. In general, the category $\der\coh^{G \times \mathbb{G}_m}(\spr^P)$ has no full exceptional collection since the variety $\spr^P$ is singular. Furthermore, the perverse-coherent t-structure does not apply because the $G$-orbits on $\spr^P$ do not satisfy the needed conditions as in \cite{ArB}. However, we are able to prove that the $\upmu{P}$-pushforward of (a subset of) the exceptional collection on $\spr$ is at least quasi-exceptional. Then arguments in \cite{BEZ2} carry-over in a straight-forward way to produce a t-structure.

\begin{thm}
There is a unique $t$-structure on $\der\coh^{G \times \mathbb{G}_m}(\spr^P)$, called the exotic t-structure, 
such that $\upmu{P*}$ takes exotic sheaves on $\spr$ to exotic sheaves on $\spr^P$. \end{thm}

As with the cases of $\spr$ and $\nilp$, our construction of the t-structure implies versions of properties (1)-(3) hold for exotic sheaves on $\spr^P$. Moreover, we strengthen property (3). For the Springer resolution $\spr$, the heart $\mathrm{ExCoh}^{G \times \mathbb{G}_m}(\spr)$ is known to be \textit{graded highest weight} by \cite{ab, MR}. In particular, this means there exist collections of objects called standards and costandards satisfying certain $\mathrm{Ext}$-vanishing conditions.  On the other hand, for the nilpotent cone $\nilp$, Minn-Thu-Aye proves the heart is \textit{graded properly stratified} in \cite{M}. This is slightly weaker than being highest weight. The difference between the two settings lies in having \textit{proper standard objects} that are not \textit{true standard objects} (in the highest weight case these classes are one and the same). We prove that exotic sheaves on $\spr^P$ also have this additional structure. 

\begin{thm}
The heart of the exotic t-structure $\mathrm{ExCoh}^{G \times \mathbb{G}_m}(\spr^P)$ is graded properly stratified. \end{thm}

\subsection{Relation with constructible sheaves} The definitions of highest weight and properly stratified categories imply that there exists classes of \textit{tilting} objects $\tilt(\spr)$ and $\tilt(\nilp)$. Moreover, there is an equivalence of additive categories between $\tilt(\spr)$ and Iwahori constructible \textit{parity sheaves} on the affine Grassmannian of the Langlands dual group $\gr_{G^\vee}$ \cite{abg, ARd, MR2}. There is also an equivalence of additive categories between $\tilt(\nilp)$ and \textit{parity sheaves} on $\gr_{G\vee}$ that are constructible along spherical orbits \cite{g, ARd2}. We expect but do not prove a similar result holds for $\tilt(\spr_P)$: associated to $P\subset G$ is a parahoric $\mathscr{P}^\vee\subset G^\vee(\mathbb{C}[[t]])$ such that $\tilt(\spr^P)$ is equivalent to parity sheaves on $\gr_{G^\vee}$ that are constructible with respect to $\mathscr{P}^\vee$-orbits. This would imply that the subcategory of perfect complexes $\der_{\mathrm{perf}}\coh^{G \times \mathbb{G}_m}(\spr^P)$ is equivalent to the mixed modular derived category $\dmix_{(\mathscr{P})}(\gr_{G^\vee})$ as defined in \cite{AR2}. 

\subsection{Other exotic t-structures} As already mentioned in this introduction, the bounded derived categories of equivariant coherent sheaves on both the nilpotent cone and the Springer resolution admit an exotic t-structure. Another related space that falls into this framework is the cotangent bundle of a partial flag variety, $T^*(G/P)$. This example has been studied in \cite{ACR}. That paper defines an exotic t-structure in this setting, and proves that the heart is highest weight. Furthermore, it is explained how this category is related to a category of Whittaker perverse sheaves on the Langlands dual affine Grassmannian. 

\subsection{Relation with algebraic geometry} For a smooth algebraic variety $X$, it is of interest to determine when $\der\coh(X)$ admits a full exceptional collection. In case $X$ is singular, this is impossible, but one might still hope to find a \textit{semi-orthogonal decomposition}. See for instance \cite{BoO, K, KKS} for some results and related discussion. Our quasi-exceptional collection gives a semi-orthogonal decomposition for $\der\coh^{G \times \mathbb{G}_m}(\spr^P)$. 

\subsection{Contents} In Section \ref{sec:background}, we collect the necessary notation and recall needed results for defining the exotic t-structure on the Springer resolution. In Section \ref{sec:qexcp}, we prove the pushforward of the exceptional collection on $\spr$ to the partial resolution defines a quasi-exceptional collection (see Proposition \ref{prop:cosqex}). Together with the dual collection, this allows definition of the exotic t-structure, see Theorem \ref{thm:exotict-str}. In Section \ref{Sec: heart}, we recall the notion of a graded properly stratified category and prove Theorem \ref{thm:propstr}, that the category $\mathrm{ExCoh}^{G \times \mathbb{G}_m}(\spr^P)$ is graded properly stratified.

\subsection{Acknowledgements} We would like to thank Pramod Achar for helpful discussions regarding the content of this paper.

\section{Background }\label{sec:background}

\subsection{Notation}\label{subsec:notation}
Fix an algebraically closed field $\Bbbk$, and let $G$ be a connected, reductive group over $\Bbbk$. Fix a maximal torus and ``negative" Borel subgroup $T \subseteq B \subseteq G$.  Let $\Phi$ denote the set of roots, and $\Sigma$ the set of simple roots. Let $\wts$ denote the character lattice of $T$, and $W$ denote the Weyl group of $G$. We assume throughout that $G/\Bbbk$ satisfy 
\begin{enumerate}
\item $\Lie(G)$ admits a nondegenerate $G$-invariant bilinear form, and
\item the derived subgroup of $G$ is simply connected.
\end{enumerate}

Note that condition (1) is implied by $\Bbbk$ having \textit{very good} characteristic with respect to $G$ by \cite[Proposition 2.5.12]{L}. Condition (2) implies existence of $\varsigma\in\wts$ with the property that $\la\varsigma, \alpha^\vee\ra = 1$ for all $\alpha\in\Sigma$. These two conditions also allow us to apply results of \cite[Section 9]{AR}. (Although \cite{AR} assumes throughout that the characteristic is bigger than the Coxeter number, \cite[Remark 9.1]{AR} says our conditions (1) and (2) suffice for the results we reference.) 

\subsection{Parabolic data}For each $I\subset\Sigma$, we get a corresponding root system $\Phi_I = \Phi \cap \mathbb{Z}I$ and positive roots $\Phi^+_I = \Phi^+\cap\Phi_I$. Let $W_I\subset W$ denote the subgroup generated by simple reflections indexed by $I$, and $w_I$ denotes the longest element in $W_I$. Let $P_I\subset G$ denote the parabolic subgroup containing $B$ associated to $I$ with Lie algebra $\mathfrak{p}_I$, so that
\[\mathfrak{p}_I = \mathfrak{b}\oplus \bigoplus_{\alpha\in\Phi^+_I} \mathfrak{g}_\alpha.\] Let $L_I$ denote the Levi factor of $P_I$ containing $T$, $\mathfrak{l}_I$ its Lie algebra, $U_I$ the unipotent radical of $P_I$, and $\mathfrak{u}_I$ its Lie algebra. We define $\wts^+_I$ as the set of characters dominant with respect to $I$. That is, $\wts^+_I = \{\lambda\in\wts | \la\lambda, \alpha^\vee\ra \geq0 \textup{ for all } \alpha\in I\}$.

Throughout we are primarily concerned with the variety defined as \[\spr^I = G \times^{P_I} (\mathfrak{p}_I \cap \mathcal{N}),\] which partially resolves the singularities of the nilpotent cone $\nilp$. We study the bounded derived category of $G\times\gm$-equivariant coherent sheaves on $\spr^I$ denoted $\der\coh^{G\times\gm}(\spr^I)$. We simplify notation for the following special cases. For $I = \Sigma$, the partial resolution identifies with the nilpotent cone $\spr^{\Sigma} = \nilp$. In case $I = \emptyset$, we get the Springer resolution $\spr$. 

Define an action of the multiplicative group $\gm$ on $\nilp$ by letting $z\cdot x = z^{-2}x$ for $x\in\nilp$. By restricting the $\gm$ action on $\nilp$, we get an action on $\mathfrak{p}_I \cap \mathcal{N}$ for any $I$, hence on the partial resolution $\spr^I$ as well. We also regard the flag varieties $G/B$ and $G/P_I$ as $G\times\gm$ variety by letting $\gm$ act trivially. We have a \textit{twist the grading} functor 
\[\la 1\ra : \der\coh^{G\times\gm}(\spr^I)\rightarrow \der\coh^{G\times\gm}(\spr^I)\] defined by tensoring with the tautological 1-dimensional $\gm$-bundle. 

The partial resolution factors the Springer resolution, so we have natural $G\times\gm$-equivariant maps 
\[\upmu{I}:\spr \rightarrow \spr^I \qquad \text{and} \qquad \mu_I:\spr^I \rightarrow \mathcal{N}.\] We also use natural maps as indicated in the commutative diagram
 \begin{equation}\label{cd:maps}\begin{tikzcd}
  \spr \arrow[r, "\pi"] \arrow[d, "\upmu{I}"]
    & G/B \arrow[d, "p"] \\
  \spr^I \arrow[r, "\pi_I"]
&G/P_I .\end{tikzcd}\end{equation} We denote by \[\vb_{G/B}: \rep(B)\rightarrow\coh^G(G/B) \textup{ and }\vb_{G/P_I}: \rep(P_I)\rightarrow\coh^G(G/P_I) \] the functors that assign the associated vector bundle to any representation, and note that each is an equivalence of categories. We will usually regard the output as being $G\times\gm$-equivariant where $\gm$ acts trivially. For each $\lambda\in\wts$, we get an associated $G\times\gm$ equivariant line bundle on $\spr$, denoted $\str_{\spr}(\lambda) := \pi^*(\vb_{G/B}(\Bbbk_\lambda))$. Let $\str_{\spr^I}$ denote the structure sheaf of $\spr^I$. 

\subsection{Orders on $\wts$} The usual partial order on $\wts$ is denoted by $\wtles$, and we will also use the partial order $\wtles_I$. That is, \[\lambda \wtles\mu \textup{ iff } \mu-\lambda\in\mathbb{Z}_{\geq0}\Phi^+ \textup{ and } \lambda \wtles_I\mu \textup{ iff } \mu-\lambda\in\mathbb{Z}_{\geq0}\Phi_I^+.\] For each $\lambda\in\wts$, let $\dm(\lambda)$ be the unique element in $W\lambda\cap\wts^+$, and let $\dm_I(\lambda)$ be the unique element in $W_I\lambda\cap\wts_I^+$. For $\lambda\in\wts$, we define a subset $\conv(\lambda)\subset\wts$ by the condition \[\mu\in\conv(\lambda) \textup{ if and only if } \dm(\mu)\wtles \dm(\lambda),\] and set $\conv^0(\lambda) := \conv(\lambda)\backslash W\lambda$. Similarly, define the subset $\conv_I(\lambda)\subset\wts$ by \[\mu\in\conv_I(\lambda) \textup{ if and only if }\dm_I(\mu)\wtles \dm_I(\lambda),\] and let $\conv^0_I(\lambda) := \conv_I(\lambda)\backslash W_I\lambda$. 
In order to define our exotic t-structure, we use a different ordering of $\wts$ whose definition is related to the Bruhat order on the affine Weyl group. We define $\leq$ on $\wts$ as a completion of the Bruhat order chosen to be ``compatible" with restriction to $\wts^+_I$. It is explained how to make such a choice in \cite[Subsection 9.4]{AR}; note that in loc. cit., they use $\leq'$ to denote this order and reserve $\leq$ for the partial order defined using Bruhat order. Moreover, note that $\lambda\in\conv^0(\mu)$ or $\lambda\in\conv_I^0(\mu)$ implies $\lambda\leq \mu$ by \cite[Lemma 9.15]{AR} and that $(\wts^+_I, \leq)$ is isomorphic as an ordered set to $(\mathbb{Z}_{\geq 0}, \leq)$.


\subsection{The Springer resolution}
From now on, we assume $I\neq\varnothing$. More specifically, we do not reprove existence of the exotic t-structure for the Springer resolution, but instead show how its existence and properties implies existence of an exotic t-structure for the partial resolution. For a subset $\mathcal{S}\subseteq \wts$, let $\mathrm{D}_{\mathcal{S}}$ denote the full triangulated subcategory of $\der\coh^{G\times\gm}(\spr)$ generated by line bundles $\str_{\spr}(\mu)\la m \ra$ with $\mu\in\mathcal{S}$ and $m\in\mathbb{Z}$. Given $\lambda\in\wts$, we may define the sets $\mathcal{S} = \{\mu | \mu\leq\lambda \}$ and $\mathcal{S}' = \{\mu | \mu<\lambda \}$. In these cases, we use the notation $\mathrm{D}_{\leq \lambda} = \mathrm{D}_{\mathcal{S}}$ and $\mathrm{D}_{< \lambda}= \mathrm{D}_{\mathcal{S}'}$.

For each $\lambda\in\wts$, let $\delta_\lambda$ denote the minimal length of an element $w\in W$ such that $w(\lambda)\in\wts^+$. The following proposition summarizes the key properties that allow one to define the exotic t-structure on $\der\coh^{G\times\gm}(\spr)$. See \cite[Section 2.1.5]{BEZ}. See also \cite[2.3, 2.4]{MR} where positive characteristic is considered. For the $\gm$-action, we normalize our objects as in \cite[Proposition 9.16]{AR}.
\begin{prop}\label{prop:excosstd}
There are objects $\excos(\lambda), \exstd(\lambda)$ in $\der\coh^{G\times\gm}(\spr)$ that are uniquely determined (up to isomorphism) by the following two properties:
\begin{enumerate}
\item there exist distinguished triangles
\begin{equation}\label{eq:sprexcos}\mathcal{F}\rightarrow \str_{\spr}(\lambda)\la-\delta_{\lambda}\ra\rightarrow \excos(\lambda)\rightarrow\end{equation}
\begin{equation}\label{eq:sprexstd}\exstd(\lambda) \rightarrow \str_{\spr}(\lambda)\la-\delta_{\lambda}\ra\rightarrow\mathcal{F}'\rightarrow\end{equation}
with $\mathcal{F}\in\mathrm{D}_{\conv^0(\lambda)}, \mathcal{F'} \in \mathrm{D}_{< \lambda}$, and
\item for all $\mathcal{G}\in  \mathrm{D}_{< \lambda}$, we have $\Hom(\mathcal{G}, \excos(\lambda)) = \Hom(\exstd(\lambda), \mathcal{G}) = 0$.
\end{enumerate}
\end{prop}

The objects $(\excos(\lambda), \lambda\in\wts)$ form a graded exceptional collection in the sense of \cite[2.3]{MR}. Moreover,  \cite[Proposition 2]{BEZ2} proves they define the \textit{exotic} t-structure on $\der\coh^{G\times\gm}(\spr)$. The heart is a graded highest weight category. Let $\exirr(\lambda)$ denote the image of the natural map $\exstd(\lambda)\rightarrow\excos(\lambda)$. Then all irreducible exotic sheaves are of the form $\exirr(\lambda)\la m\ra$, $\lambda\in\wts, m\in\mathbb{Z}$.
For any $m\in\mathbb{Z}$, we call $\excos(\lambda)\la m\ra$ a \textit{costandard} exotic sheaf and $\exstd(\lambda)\la m\ra$ a \textit{standard} exotic sheaf. 

\section{Case of a partial resolution}\label{sec:qexcp}

\subsection{Generalized Andersen--Jantzen sheaves} 
For each $\lambda\in\wts$, let $A_I(\lambda) = \upmu{I*} \str_{\spr}(\lambda)$. We call these objects \textit{generalized Andersen--Jantzen sheaves}. For a subset $\mathcal{S}\subseteq \wts^+_I$, let $\mathrm{D}^I_{\mathcal{S}}$ denote the full triangulated subcategory of $\der\coh^{G\times\gm}(\spr^I)$ generated by generalized Andersen--Jantzen sheaves $A_I(\mu)\la m \ra$ with $\mu\in\mathcal{S}$ and $m\in\mathbb{Z}$. Given $\lambda\in\wts^+_I$, we may define the sets $\mathcal{S} = \{\mu\in\wts^+_I | \mu\leq\lambda \}$ and $\mathcal{S}' = \{\mu\in\wts^+_I | \mu<\lambda \}$. In these cases, we use the notation $\mathrm{D}^I_{\leq \lambda} = \mathrm{D}^I_{\mathcal{S}}$ and $\mathrm{D}^I_{< \lambda}= \mathrm{D}^I_{\mathcal{S}'}$. The following lemma shows that $\mathrm{D}^I_{<\lambda}$ already contains $A_I(\mu)$ for any $\mu\in\wts$ with $\dm_I(\mu)<\lambda$. Our proof is nearly identical to \cite[Lemma 5.3]{A2} suitably modified to handle that $\spr^I$ is not (generally) an affine variety.

 \begin{lemma}\label{lem:waj} Let $\lambda, \mu\in\wts$ satisfy $W_I \lambda = W_I\mu$ and $\mu\leq\lambda$. Then there is a distinguished triangle 

\[A_I(\mu)\la-2\ell \ra\rightarrow A_I(\lambda) \rightarrow \mathcal{H}\rightarrow\]
where $\ell$ is the length of the minimal $w\in W_I$ so that $w(\mu) = \lambda$ and $\mathcal{H}\in\mathrm{D}^I_{\conv^0_I(\lambda)}$.
\end{lemma}

\begin{proof} Clearly it suffices to prove the case $\mu = s\lambda$ for simple reflection $s\in W_I$. Suppose $s$ corresponds to the simple root $\alpha\in I$, and let $n = \la \lambda, \alpha^\vee \ra >0$. $P_\alpha$ denotes the minimal parabolic subgroup corresponding to $\alpha$. Recall $\varsigma\in\wts$ has the property that $\la\varsigma, \beta^\vee\ra = 1$ for all simple roots $\beta$. Since $\la\varsigma-\alpha, \alpha^\vee\ra = -1$, \cite[II.5.2(b)]{J} implies $R\ind{P_\alpha}{B}\bbk_{\varsigma - \alpha} = 0$. Define the $B$-module $Q := \bbk_{\varsigma - \alpha}\otimes \res{P_\alpha}{B}\ind{P_\alpha}{B} \bbk_{\lambda-\varsigma}.$ Using the derived tensor identity along with the above vanishing, we have  \begin{equation}R\ind{P_\alpha}{B} Q\cong R\ind{P_\alpha}{B}\bbk_{\varsigma - \alpha} \otimes^L R\ind{P_\alpha}{B}\bbk_{\lambda - \varsigma} = 0.\label{eqvan2}\end{equation}

 Now we study the $B$-module structure of the relevant representations. Since $\la \lambda -\varsigma, \alpha^\vee \ra = n-1$, the weights of $\ind{P_\alpha}{B} \bbk_{\lambda-\varsigma}$ are $\lambda - \varsigma, \lambda - \varsigma - \alpha, \ldots,  \lambda - \varsigma - (n-1)\alpha$. Thus, the weights of $Q$ are $\lambda - \alpha, \lambda - 2\alpha, \ldots,  \lambda - n\alpha = s\lambda.$ In particular, there are short exact sequences of $B$-modules \[0\rightarrow \bbk_{s\lambda}\rightarrow Q \rightarrow K_1\rightarrow 0 \textup{ and }0\rightarrow K_2\rightarrow Q\otimes \bbk_{\alpha} \rightarrow \bbk_{\lambda}\rightarrow 0,\] and weights of $K_1$ and $K_2$ are in $\conv^0_I(\lambda)$. Application of $F = \upmu{I*}\circ p^*\circ\vb_{G/B}$ to the above sequences (and twisting the second) give distinguished triangles \[A_I(s\lambda)\rightarrow F(Q) \rightarrow \mathcal{K}_1\rightarrow \textup{ and }\mathcal{K}_2\rightarrow F(Q\otimes \bbk_{\alpha})\la2\ra \rightarrow A_I(\lambda)\la2\ra\rightarrow\] with $\mathcal{K}_1,\mathcal{K}_2\in\mathrm{D}^I_{\conv^0_I(\lambda)}.$ Hence, the lemma holds once we show there is an isomorphism \begin{equation}\label{eq:iso}F(Q)\cong F(Q\otimes \bbk_{\alpha})\la2\ra.\end{equation}

Let $\mathfrak{u}_\alpha$ denote the Lie algebra of the unipotent radical of $P_\alpha$, and let $\spr_{\alpha} = G\times^B \mathfrak{u}_\alpha$. There is a closed embedding $i: \spr_{\alpha}\rightarrow \spr$ and related short exact sequence in $\coh^{G\times\gm}(\spr)$ \[0\rightarrow\str_{\spr}(\alpha)\la2\ra\rightarrow\str_{\spr}\rightarrow i_*\str_{\spr_\alpha}\rightarrow 0.\] Under equivariant induction equivalence $\coh^{G\times\gm}(\spr)\cong\coh^{B\times\gm}(\mathfrak{u})$, this corresponds to the short exact sequence  of $B\times\gm$ equivariant $\bbk[\mathfrak{u}]$ modules \[0\rightarrow\bbk[\mathfrak{u}]\otimes\bbk_\alpha\la2\ra\rightarrow\bbk[\mathfrak{u}]\rightarrow \bbk[\mathfrak{u}_\alpha]\rightarrow 0.\] Tensor the above short exact sequence by $Q$ and apply equivariant induction equivalence $\coh^{G\times\gm}(\spr)\cong\coh^{B\times\gm}(\mathfrak{u})$ to obtain a distinguished triangle  \begin{equation}\pi^*\vb_{G/B}(Q\otimes\bbk_\alpha)\la 2\ra\rightarrow \pi^*\vb_{G/B}(Q)\rightarrow i_*\str_{\spr_\alpha}\otimes \pi^*\vb_{G/B}(Q)\rightarrow. \label{eqvan4}\end{equation} In order to get isomorphism \eqref{eq:iso}, we will show $\upmu{I*}( i_*\str_{\spr_\alpha}\otimes \pi^*\vb_{G/B}(Q)) = 0$. Now, $\spr^I$ is not affine, but it is a fiber bundle  with affine fibers over $G/P_I$. Consider the set \[\Pi = \{\lambda\in \wts | \la\lambda, \alpha^\vee\ra = 0 \textup{ for all } \alpha\in I \textup{ and } \la\lambda, \alpha^\vee\ra > 0 \textup{ for all } \alpha\in S/I \};\] this is the set of weights that give rise to ample line bundles on $G/P_I$, see \cite[II.4.4, Remark (1)]{J}. By standard arguments, it suffices to prove that \[R\Gamma(\upmu{I*} (i_*\str_{\spr_\alpha}\otimes \pi^*\vb_{G/B}(Q))\otimes_{\spr^I} \str_{\spr^I}(\gamma)) = 0\] for $\gamma\in\Pi$ is sufficiently large. Of course, $R\Gamma\circ\upmu{I*} = R\Gamma\circ \pi_*$. The projection formula implies $\pi_*(i_*\str_{\spr_\alpha}\otimes \pi^*\vb_{G/B}(Q)\otimes \str_{\spr}(\gamma))\cong \vb_{G/B}(\bbk[\mathfrak{u}_\alpha]\otimes Q\otimes \bbk_{\gamma}).$ Moreover, $R\Gamma\vb_{G/B}(\bbk[\mathfrak{u}_\alpha]\otimes Q\otimes \bbk_{\gamma})\cong R\ind{G}{B}(\bbk[\mathfrak{u}_\alpha]\otimes Q\otimes \bbk_{\gamma})$.

Using tensor identity and \eqref{eqvan2}, we obtain the vanishing $R\ind{P_\alpha}{B}(\bbk[\mathfrak{u}_\alpha]\otimes Q\otimes \bbk_\gamma)  = \bbk[\mathfrak{u}_\alpha]\otimes R\ind{P_\alpha}{B}(Q) \otimes \bbk_\gamma = 0\label{eqvan2}$ for all $\gamma\in\Pi$. This completes the argument.\end{proof}

Next we explain how to get generalized Andersen--Jantzen sheaves from Andersen--Jantzen sheaves on the nilpotent cone of the Levi. For the remainder of this subsection, we let $P = P_I$ with Levi decomposition $P=LU_I$. Let $\nilp_L$ denote the nilpotent cone of the Levi factor $L$. The Levi decomposition yields $\nilp\cap\frak{p} = \nilp_L+\frak{u}_I$.  Let $m: \nilp\cap\frak{p}\rightarrow \nilp_L$ denote the obvious map. Note that the category $\der\coh^{P\times\gm}(\nilp_L)$ is generated as a triangulated category by it's own Andersen--Jantzen sheaves, $(A^{(L)}(\lambda)\la m\ra, \lambda\in\wts^+_I, m\in\mathbb{Z})$ by \cite[Lemma 5.9]{A2}. Note also that that argument only shows the $L$-equivariant case, but the same argument applies to the $P$-equivariant case since the $L$-equivariant irreducible perverse coherent sheaves are the same as the $P$-equivariant ones. Also, \cite[Lemma 5.3]{A2} shows that only those labelled by weights dominant for $L$ are needed.

\begin{lemma}\label{lem:inducedaj} \begin{enumerate}
\item For all $\lambda\in\wts$, $m^*(A^{(L)}(\lambda))$ corresponds to the generalized Andersen--Jantzen sheaf $A_I(\lambda)$ under the equivariant induction equivalence $ \der\coh^{G\times\gm}(\spr^I)\cong\der\coh^{P\times\gm}(\nilp\cap\frak{p})$. 
\item There is an isomorphism $\upmu{I*}\str_{\spr} \cong \str_{\spr^I}$, and the sheaf $\str_{\spr^I}$ is an equivariant dualizing complex on $\spr^I$. \end{enumerate}
\end{lemma}

\begin{proof}
For part (1), we will use the maps as defined in the following commutative diagram. 

\[
  \begin{tikzcd}
    \nilp\cap\frak{p} \arrow{r}{i_2} \arrow{d}{m} & G \times^P (\nilp\cap\frak{p}) \arrow{d}{\hat m} \\
    \nilp_L \arrow{r}{i_1} & G \times^{P} \nilp_{L} 
  \end{tikzcd}
\]

The maps $m$ and $\hat m$ are induced from the quotient $\mathfrak{p}_I\rightarrow \mathfrak{l}_I$, and are affine bundle maps. The inclusions $i_1$ and $i_2$ are $P$-equivariant, and each induces the induction equivalence
\[(i_1)^*: \coh^G (G \times^{P} \nilp_{L}) \xrightarrow{\sim} \coh^{P} (\nilp_{L}) \textup{ and } (i_2)^*: \coh^G (G \times^{P} \nilp\cap\frak{p}) \xrightarrow{\sim} \coh^{P} (\nilp\cap\frak{p}).\]
Denote the inverse equivalences by $(i_1^*)^{-1}$ and $(i_2^*)^{-1}$. We clearly have an isomorphism of functors $\hat m^*(i_1^*)^{-1} \cong (i_2^*)^{-1}m^*$. Therefore it suffices to show an isomorphism of objects in $\der \coh^G( \spr^I)$,
\[A_I(\lambda) \cong \hat m^*(i_1^*)^{-1}(A^{(L)}(\lambda)).\] To get this isomorphism, we must commute a few more induction equivalences with appropriate sheaf functors. \[
  \begin{tikzcd}
    \spr_L\cong P\times^B(\frak{u}/\frak{u}_I) \arrow{r}{i_3} \arrow{d}{\nu} & G \times^B (\frak{u}/\frak{u}_I) \arrow{d}{\hat\nu} \hspace{.5cm} & P\times^B(\frak{u}/\frak{u}_I) \arrow{r}{i_3} \arrow{d}{q_1} & G \times^B (\frak{u}/\frak{u}_I) \arrow{d}{\hat q_1} \\
    \nilp_L \arrow{r}{i_1} & G \times^{P} \nilp_{L} \hspace{.5cm} &P/B \arrow{r}{i_4} & G \times^{P} P/B 
  \end{tikzcd}
\]

Note that the leftmost arrow identifies P-equivariantly with the Springer resolution for $\nilp_L$, and there is a $P$-equivariant isomorphism $P/B$ with $L$'s flag variety $L/(L\cap B)$. Moreover, we have isomorphisms $G\times^P(P/B)\cong G/B$, and $G \times^B (\frak{u}/\frak{u}_I)\cong G\times^P (P\times^B \frak{u}/\frak{u}_I)$. Using the same reasoning as above, we get isomorphisms of functors $(i_1^*)^{-1}\nu_*\cong \hat\nu_*(i_3^*)^{-1}$ and $(i_3^*)^{-1}q^*_1 \cong \hat q_1^*(i_4^*)^{-1}$. This yields the following string of isomorphisms:

\begin{align*} \hat m^*(i_1^*)^{-1}(A^{(L)}(\lambda)) &\cong  \hat m^*(i_1^*)^{-1}\nu_*(\str_{\spr_{L}}(\lambda))\\
										     &\cong  \hat m^*\hat\nu_*(i_3^*)^{-1}q_1^*(\vb_{P/B}(\Bbbk_{\lambda}))	\\	
										     &\cong  \hat m^*\hat\nu_*\hat q_1^*(i_4^*)^{-1}(\vb_{P/B}(\Bbbk_{\lambda}))\\		
										     &\cong  \hat m^*\hat\nu_*\hat q_1^*(\vb_{G/B}(\Bbbk_{\lambda})).
										     	\end{align*}
Next we must apply flat base change theorem with the maps in the next diagram.

\begin{equation}\label{eq:bcdiagram}
  \begin{tikzcd}
    G \times^B \mathfrak{u} \arrow{r}{q_2} \arrow{d}{\upmu{I}} & G \times^B (\mathfrak{u}/\mathfrak{u}_I) \arrow{d}{\hat\nu} \\
    G \times^{P} \nilp\cap\frak{p} \arrow{r}{\hat{m}} & G \times^{P} \nilp_{L} 
  \end{tikzcd}
\end{equation}

Flat base change theorem says $\hat m^*\hat\nu_*\cong \upmu{I*} q_2^*$. This combined with the final isomorphism above finish the proof. 

For the second statement, we recall that $\nu_*\str_{\spr_L}\cong \str_{\nilp_L}$. This follows from \cite[Lemmas 3.4.2 and 5.1.1]{BK} which imply in particular that $\str_{\spr_L}$ and $\str_{\nilp_L}$ are equivariant dualizing complexes on their corresponding varieties. Now, equivariant induction equivalence implies we also have an isomorphism $\hat\nu_*\str_{G\times^B(\mathfrak{u}/\mathfrak{u}_I)}\cong \str_{G\times^P\nilp_L}$. Flat base change with diagram \eqref{eq:bcdiagram} shows an isomorphism $\upmu{I*}\str_{\spr}\cong\str_{\spr^I}$, and  \cite[Lemmas 3.4.2 and 5.1.1]{BK} finish the argument.\end{proof}

\begin{prop}\label{prop:pushfor} The category $\der\coh^{G\times\gm}(\spr^I)$ is generated as a triangulated category by the image of the functor $\upmu{I*}: \der\coh^{G\times\gm}(\spr)\rightarrow \der\coh^{G\times\gm}(\spr^I)$. 

\end{prop}

\begin{proof} Suppose $P = P_I$ has Levi decomposition $P=LU_I$. Let $\nilp_L$ denote the nilpotent cone of the Levi factor $L$. The space $\nilp\cap\frak{p} = \nilp_L+\frak{u}_I$ has a P-stable stratification by subspaces $C+\frak{u}_I$ where $C$ is an $L$-orbit in $\nilp_L$, and so we also have that $\spr^I$ is stratified by subspaces $X_C = G\times^P (C+\frak{u}_I)$. We will prove the proposition by induction on support with respect to this stratification. 

For the base case, let $\mathcal{F}$ be an object in $\der\coh^{G\times\gm}(\spr^I)$ so that the support of $\mathcal{F}$ is contained in $X_{\{0\}} = G\times^P \frak{u}_I$. Note that $X_{\{0\}}$ identifies with the variety denoted $\spr_I$ in \cite[Section 9.1]{AR}. According to Lemma 9.3 in loc. cit., the category $\der\coh^{G\times\gm}(\spr_I)$ is generated by (the restriction of) the collection of vector bundles $(\mathcal{V}_I(\lambda)\la n\ra$, $\lambda\in\dom_I$ and $n\in\mathbb{Z}$) (see Subsection \ref{subsec:vbund} for definition). Let $i : X_{\{0\}}\rightarrow \spr^I$ denote the closed inclusion. The object $i_*i^*(\mathcal{V}_I(\lambda)\la n\ra)$ can be obtained as the pushforward of an appropriately defined vector bundle along the composition $G\times^B\mathfrak{u}_I \hookrightarrow \spr \rightarrow \spr^I$. Hence the base case holds. 

Now, assume $\mathcal{F}$ is an object in $\der\coh^{G\times\gm}(\spr^I)$ with support contained in $\overline{X_{C}} = G\times^P(\overline{C} + \frak{u}_I)$. To prove the proposition in general, we need to produce an object $\tilde{\mathcal{G}}$ in $\der\coh^{G\times\gm}(\spr)$ and a morphism $\phi: \mathcal{F}\rightarrow \upmu{I*}\tilde{\mathcal{G}}$ so that $\mathrm{cone}(\phi)$ has support contained within $X_{\partial C} = G\times^P(\overline{C}/C + \frak{u}_I)$. This follows in the same manner as \cite[Lemma 7]{BEZ2}. Note that the argument there assumes varieties are defined over characteristic 0, and uses the Jacobson--Morozov--Deligne filtration to obtain a resolution for $\overline{C}$, but this may be replaced by action of an associated cocharacter. See for instance \cite[Proposition 5.9 and 8.8 (II)]{J0}. \end{proof}

\begin{cor}\label{lem:ajgen} The objects $(A_I(\lambda)\la m\ra, \lambda\in\wts^+_I, m\in\mathbb{Z})$ generate $\der\coh^{G\times\gm}(\spr^I)$ as a triangulated category.
\end{cor}

\begin{proof} This follows by combining \cite[Corollary 5.8]{A2}, Proposition \ref{prop:pushfor}, and Lemma \ref{lem:waj}.\end{proof}

\subsection{Proper (co)standards} 
 
 For any $\lambda\in\wts$, let $\delta^I_{\lambda}$ denote the minimal length of an element $w\in W_I$ such that $w(\lambda)\in\wts^+_I$.
 \begin{lemma}\label{lem:stdtostd} The functor $\upmu{I*}: \der\coh^{G\times\gm}(\spr)\rightarrow \der\coh^{G\times\gm}(\spr^I)$ satisfies
\begin{enumerate}
\item $\upmu{I*} \exstd(\lambda) \la \delta^I_\lambda \ra \cong \upmu{I*} \exstd(\dm_I(\lambda))$ 
\item $\upmu{I*} \excos(\lambda) \la -\delta^I_\lambda \ra \cong \upmu{I*} \excos(\dm_I(\lambda))$ 
\end{enumerate}
\end{lemma}
\begin{proof} Our proof is similar to \cite[Lemma 8]{BEZ}. Parts (1) and (2) are similar; we prove (2). Suppose that $s\in W_I$ is a simple reflection corresponding to simple root $\alpha\in I$, and $s\lambda\wtless\lambda$. It is sufficient to show $\upmu{I*} \excos(\lambda) \cong \upmu{I*} \excos(s\lambda)\la -1\ra$. \cite[Proposition 9.19 (4)]{AR} gives a distinguished triangle in $\der\coh^{G \times\gm}(\spr)$ : \[\excos(\lambda)\rightarrow \excos(s\lambda)\la-1\ra\rightarrow \Psi_s(\excos(\lambda))\la-1\ra\rightarrow.\] (Note that the functor $\Psi_s: \der\coh^{G \times\gm}(\spr)\rightarrow \der\coh^{G \times\gm}(\spr)$ is written as a composition of two functors in \cite{AR}: $\Pi_s: \der\coh^{G \times\gm}(\spr)\rightarrow \der\coh^{G \times\gm}(T^*G/P_\alpha)$ and $\Pi^s: \der\coh^{G \times\gm}(T^*G/P_\alpha)\rightarrow \der\coh^{G \times\gm}(\spr)$.) The result follows once we show $\upmu{I*}\circ\Psi_s = 0$. Recall that $\mathfrak{u}_\alpha$ denotes the Lie algebra of the unipotent radical of $P_\alpha$. The map $\upmu{I}$ factors as $\spr\rightarrow G\times^{P_\alpha}(\mathfrak{p}_I \cap \mathcal{N})\rightarrow\spr^I$. Moreover, the compositions of the natural maps $G\times^B\mathfrak{u}_\alpha\rightarrow\spr\rightarrow G\times^{P_\alpha}(\mathfrak{p}_I \cap \mathcal{N})\rightarrow\spr^I$ and $G\times^B\mathfrak{u}_\alpha\rightarrow G\times^{P_\alpha}\mathfrak{u}_\alpha\rightarrow G\times^{P_\alpha}(\mathfrak{p}_I \cap \mathcal{N})\rightarrow\spr^I$ are equal. On the other hand, the map $G\times^B\mathfrak{u}_\alpha\rightarrow G\times^{P_\alpha}\mathfrak{u}_\alpha$ has fibers isomorphic to $\mathbb{P}^1$. Now, for any $\mathcal{F}$ in $\der\coh^{G \times\gm}(\spr)$, the object $\Psi_s(\mathcal{F})$ is an extension of an object from $G\times^B\mathfrak{u}_\alpha$, and along the fibers of $G\times^B\mathfrak{u}_\alpha\rightarrow G\times^{P_\alpha}\mathfrak{u}_\alpha$ is isomorphic to a sum of copies of $\str_{\mathbb{P}^1}(-1)[n]$ for various $n\in\mathbb{Z}$. Hence, the pushforward along $\upmu{I}$ vanishes. 
\end{proof}

\begin{defn}For each $\lambda\in\wts^+_I$, we define objects in $\der\coh^{G\times\gm}(\spr^I)$ by \[\expcos_I(\lambda) = \upmu{I*} \excos(\lambda) \textup{ and }\expstd_I(\lambda) = \upmu{I*} \exstd(w_I\lambda)\la-\delta^I_{w_I\lambda}\ra\cong\upmu{I*} \exstd(\lambda)\la-2\delta^I_{w_I\lambda}\ra.\] These objects (and their $\gm$-twists) are called \textit{proper costandard} and \textit{proper standard} respectively. (This terminology will be justified in Section \ref{Sec: heart}. We are reserving the notation $\excos_I(\lambda)$ and $\exstd_I(\lambda)$ for another class of objects.) 
\end{defn}
\begin{lemma}\label{lem:dtri} The proper costandard and proper standard objects fit into distinguished triangles 
\[\mathcal{G}\rightarrow A_I(\lambda)\la-\delta_{\lambda}\ra\rightarrow \expcos_I(\lambda)\rightarrow \hspace{.1in}\textup{ and }\hspace{.1in}\expstd_I(\lambda) \rightarrow A_I(\lambda)\la-\delta_{\lambda}\ra\rightarrow\mathcal{G}'\rightarrow\]
satisfying $\mathcal{G}, \mathcal{G'} \in \mathrm{D}^I_{< \lambda}$.
\end{lemma}
\begin{proof}
The first distinguished triangle is obtained by pushing forward distinguished triangle \eqref{eq:sprexcos} for weight $\lambda$. To see the restriction on the cocone of the first triangle, first note that \cite[Proposition 9.19 (2)]{AR} says that in triangle \eqref{eq:sprexcos}, $\mathcal{F} \in \mathrm{D}_{\conv^0(\lambda)}$. This along with Lemma \ref{lem:waj} imply $\mathcal{G} \in \mathrm{D}^I_{< \lambda}$.


For the second triangle, we will apply the octahedral axiom. Note that since $\lambda\in\wts^+_I$, we have $\delta^I_{w_I\lambda}+\delta_{\lambda} = \delta_{w_I\lambda}$, so Lemma \ref{lem:waj} gives distinguished triangle 
\begin{equation}\label{eq:ajtri}A_I(w_I\lambda)\la-\delta_{w_I\lambda}-\delta^I_{w_I\lambda}\ra\rightarrow A_I(\lambda)\la-\delta_\lambda \ra \rightarrow \mathcal{G}_1\rightarrow\end{equation} with $\mathcal{G}_1\in\mathrm{D}^I_{\conv^0_I(\lambda)}$. We also use the triangle \begin{equation}\label{eq:alttri}\expstd_I(\lambda) \rightarrow A_I(w_I\lambda)\la-\delta_{w_I\lambda}-\delta^I_{w_I\lambda}\ra\rightarrow\mathcal{G}_2\rightarrow\end{equation} which is obtained by pushing forward distinguished triangle \eqref{eq:sprexstd} for weight $w_I\lambda$, then applying the twist $\la -\delta^I_{w_I\lambda}\ra$. We automatically get that $\mathcal{G}_2$ is in the triangulated category generated by $A_I(\gamma)\la m \ra$ with $\gamma < w_I\lambda$ and $m\in\mathbb{Z}$. Of course, $\gamma < w_I\lambda$ implies $W_I\gamma \neq W_I\lambda$, so we may apply \cite[Subsection 9.4, (9.9)]{AR} to see that we have  $\dm_I(\gamma) < w_I\lambda$ as well. This together with Lemma \ref{lem:waj} imply $\mathcal{G}_2\in\mathrm{D}^I_{< \lambda}$. The first morphism of the distinguished triangle we desire is the composition of the two first morphisms in triangles \eqref{eq:ajtri} and \eqref{eq:alttri}.

Finally the octahedral axiom gives a distinguished triangle \[\mathcal{G}_2\rightarrow\mathcal{G}'\rightarrow\mathcal{G}_1\rightarrow\] which implies $\mathcal{G}'\in\mathrm{D}^I_{< \lambda}$ since the same is true for $\mathcal{G}_1$ and $\mathcal{G}_2$.\end{proof}

\begin{cor}\label{lem:cosgen} The objects $(\expcos_I(\lambda)\la m\ra, \lambda\in\wts^+_I, m\in\mathbb{Z})$ generate $\der\coh^{G\times\gm}(\spr^I)$ as a triangulated category. Similarly, the objects $(\expstd_I(\lambda)\la m\ra, \lambda\in\wts^+_I, m\in\mathbb{Z})$ generate $\der\coh^{G\times\gm}(\spr^I)$ as a triangulated category.
\end{cor}
\begin{proof}This follows from Lemmas \ref{lem:ajgen} and \ref{lem:dtri}.
\end{proof}

\subsection{Helpful vector bundles} \label{subsec:vbund}
For each $\mu\in\wts^+_I$, let $\mathcal{V}(\mu)$ denote the vector bundle on $\spr$ defined by $\pi^*\mathscr{S}_{G/B}(N_I(\mu))$ where $N_I(\mu)$ is the dual Weyl module for $L_I$: $N_I(\mu) = \mathrm{Ind}^{L_I}_{B\cap L_I} \bbk_\mu \cong \mathrm{Ind}^{P_I}_B \bbk_\mu$. Similarly, we let $\mathcal{V}_I(\mu)$ denote the vector bundle on $\spr^I$ defined by $\pi_I^*\mathscr{S}_{G/P_I}(N_I(\mu))\cong\upmu{I*}\mathcal{V}(\mu)$. 

\begin{remark}\label{rem:pullbacks} There are isomorphisms $\upmu{I}^*\mathcal{V}_I(\mu)\cong \mathcal{V}(\mu)$ and $\upmu{I}^!\mathcal{V}_I(\mu)\cong \mathcal{V}(\mu)$. Recall the commutative diagram \eqref{cd:maps}. Hence we have isomorphisms $\upmu{I}^*\mathcal{V}_I(\mu)\cong \upmu{I}^*\pi^*_I\vb_{G/P_I}(N_I(\mu))\cong  \pi^*p^*\vb_{G/P_I}(N_I(\mu))\cong \pi^*\vb_{G/B}(N_I(\mu))= \mathcal{V}(\mu).$

For the second identification, we use duality. We have $\upmu{I}^! \cong \vd\circ\upmu{I}^*\circ\vd$ where $\vd$ denotes Serre--Grothendieck duality. Moreover, $\str_{\spr^I}$ is a dualizing complex on $\spr^I$ by Lemma \ref{lem:inducedaj}, so $\vd = R\mathcal{H}om(-, \str_{\spr^I})$ takes a vector bundle to the dual vector bundle. 
\end{remark}

For two objects $\mathcal{F}, \mathcal{G}$, we introduce the notation $\Hom^i(\mathcal{F}, \mathcal{G}) := \Hom(\mathcal{F}, \mathcal{G}[i])$. 

\begin{lemma}\label{lem:vectorbundlehelp} Suppose $\mu, \lambda\in\wts^+_I$.
\begin{enumerate}
\item If  $\mu<\lambda$, then for all $m, i\in\mathbb{Z}$, \[\Hom^i(\mathcal{V}_I(\mu), \expcos_I(\lambda)\la m\ra) = \Hom^i(\expstd_I(\lambda), \mathcal{V}_I(\mu)\la m\ra)= 0.\] 
\item $\Hom^i(\mathcal{V}_I(\mu), \expcos_I(\mu)\la m\ra)\cong\begin{cases}
\bbk &\text{if } i=0, m = \delta_\mu\\
0 &\text{else} 
\end{cases}$  

and $\Hom^i(\expstd_I(\mu)\la m\ra, \mathcal{V}_I(\mu))\cong\begin{cases}
\bbk &\text{if } i=0, m = \delta_{w_I\mu} + \delta^I_{w_I\mu}\\
0 &\text{else} 
\end{cases}$
\item There are distinguished triangles
 \[\mathcal{V}_I(\mu)\la -\delta_\mu\ra\rightarrow \expcos_I(\mu)\rightarrow\mathcal{B}\rightarrow \hspace{.1in}\textup{ and } \hspace{.1in} \expstd_I(\mu)\rightarrow\mathcal{V}_I(\mu)\la \delta\ra\rightarrow \mathcal{B}'\rightarrow \] with $\mathcal{B}, \mathcal{B}'$ in $\mathrm{D}^I_{\leq \mu}$ and $\delta =-\delta_{w_I\mu}-\delta^I_{w_I\mu}$. 
 \end{enumerate}
\end{lemma}
\begin{proof}
(1) First, adjunction gives an isomorphism with $\Hom^i (\upmu{I}^*\mathcal{V}_I(\mu), \excos(\lambda)\la m\ra).$ The vector bundle $\upmu{I}^*\mathcal{V}_I(\mu)\cong\mathcal{V}(\mu)$ (see Remark \ref{rem:pullbacks}) is filtered by line bundles $\str_{\spr}(\nu)$ with $\nu<\lambda$ (since $\nu\in \conv_I(\mu)\subset \conv(\mu)$ and $\mu<\lambda$), hence $\mathcal{V}(\mu)\in \mathrm{D}_{<\lambda}$. Proposition \ref{prop:excosstd} part (2) yields $\Hom^i_{\der\coh^{G\times\gm}(\spr)} (\mathcal{V}(\mu), \excos(\lambda)\la m\ra) =0$, which implies $\Hom^i(\mathcal{V}_I(\mu), \expcos_I(\lambda)\la m\ra) = 0$ for all $m, i\in\mathbb{Z}$.

The second Hom-vanishing is similar. First use adjunction to get an isomorphism with $\Hom^i_{\der\coh^{G\times\gm}(\spr)} (\exstd(\lambda), \upmu{I}^!\mathcal{V}_I(\mu)\la 2\delta^I_{w_I\lambda}+m\ra).$ Remark \ref{rem:pullbacks} gives an isomorphism $\upmu{I}^!\mathcal{V}_I(\mu) \cong \mathcal{V}(\mu)$, and $ \mathcal{V}(\mu)\in \mathrm{D}_{<\lambda}$. Proposition \ref{prop:excosstd} part (2) yields $\Hom^i_{\der\coh^{G\times\gm}(\spr)} (\exstd(\lambda), \mathcal{V}(\mu)\la m\ra) =0$, which implies the claim.

(2) There is a short exact sequence of $B$ modules $0\rightarrow A\rightarrow \mathrm{N}_I(\mu)\rightarrow \bbk_\mu\rightarrow 0$ with all weights of $A$ in $\conv_I(\mu)\backslash\{\mu\}$. This yields a distinguished triangle (really a short exact sequence of coherent sheaves) in $\der\coh(\spr)$ \[\mathcal{K}\rightarrow \mathcal{V}(\mu)\rightarrow\str_{\spr}(\mu)\rightarrow\] with $\mathcal{K}\in\mathrm{D}_{<\mu}$. Recall that $\der\coh^{G\times\gm}(\spr)$ has a quotient functor \[\Pi_{<\mu} :  \der\coh^{G\times\gm}(\spr)\rightarrow \der\coh^{G\times\gm}(\spr)/\mathrm{D}_{<\mu}\] which admits left and right adjoints as in \cite[Lemma 4 (d)]{BEZ2}. Applying $\Pi_{<\mu}$ to the above distinguished triangle gives an isomorphism $\Pi_{<\mu}\mathcal{V}(\mu)\cong\Pi_{<\mu}\str_{\spr}(\mu)$. This yields the first isomorphism in 
\begin{align*} 
\Hom^i(\Pi_{<\mu} \str_{\spr}(\mu)\la m \ra,\Pi_{<\mu} \str_{\spr}(\mu)) &\cong \Hom^i(\Pi_{<\mu} \mathcal{V}(\mu)\la m\ra,\Pi_{<\mu} \str_{\spr}(\mu))\\
&\cong \Hom^i(\Pi_{<\mu} \mathcal{V}(\mu)\la m-\delta_\mu\ra,\Pi_{<\mu} \str_{\spr}(\mu)\la-\delta_\mu\ra)\\
&\cong \Hom^i(\mathcal{V}(\mu)\la m-\delta_\mu\ra,\excos(\mu))\\
&\cong \Hom^i(\mathcal{V}_I(\mu)\la m-\delta_\mu\ra,\expcos_I(\mu)).
\end{align*} The third follows by adjunction and since $\excos(\mu)\cong \Pi^R_{<\mu} \Pi_{<\mu} \str_{\spr}(\mu)\la-\delta_\mu\ra$ where $\Pi_{<\mu}^R$ denotes the right adjoint to $\Pi_{<\mu}$. The fourth isomorphism follows from $\upmu{I}^*\mathcal{V}_I(\mu)\cong \mathcal{V}(\mu)$ and adjunction. By \cite[Lemma 2 (d)]{BEZ2}, we get an isomorphism for the first term \[\Hom^i(\Pi_{<\mu} \str_{\spr}(\mu)\la m\ra,\Pi_{<\mu} \str_{\spr}(\mu)) \cong\begin{cases}
\bbk &\text{if } i=m=0\\
0 &\text{else} 
\end{cases}\] since the object $\str_{\spr}(\mu)$ is exceptional in $\der\coh^{G\times\gm}(\spr)$. A similar argument works for $\expstd_I(\mu)$. 

(3) Recall the short exact sequence of $B$ representations $0\rightarrow A\rightarrow N_I(\mu)\rightarrow \bbk_\mu\rightarrow 0$ above. Applying $\upmu{I*}p^*\mathscr{S}_{G/B}$ to the short exact sequence and twisting the $\gm$ action yields the distinguished triangle $\mathcal{V}_I(\mu)\la -\delta_\mu\ra\rightarrow A_I(\mu)\la -\delta_\mu\ra\rightarrow \mathcal{A}\rightarrow$ with $\mathcal{A}$ in $\mathrm{D}^I_{\leq \mu}$. 

The morphism $\mathcal{V}_I(\mu)\la -\delta_\mu\ra\rightarrow \expcos_I(\mu)$ is defined as the composition of the maps $\mathcal{V}_I(\mu)\la -\delta_\mu\ra\rightarrow A_I(\mu)\la -\delta_\mu\ra$ from above and $A_I(\mu)\la -\delta_\mu\ra \rightarrow \expcos_I(\mu)$ from Lemma \ref{lem:dtri}. Thus, we may apply the octahedral axiom to get the diagram:
 
 \begin{tikzpicture}[x=1.5cm,y=1.5cm]
\foreach \n/\x/\y/\t in {X/0/2/\mathcal{V}_I(\mu)\la -\delta_\mu\ra,Z/2/2/\expcos_I(\mu),W/4/2/\mathcal{C},TU/6/2/\mathcal{A}[1],Y/1/1/A_I(\mu)\la -\delta_\mu\ra,V/3/1/\mathcal{B},TY/5/1/A_I(\mu)\la -\delta_\mu\ra[1],U/2/0/\mathcal{A},TX/4/0/\mathcal{V}_I(\mu)\la -\delta_\mu\ra[1]} \node (\n) at (\x,\y) {$\t$};
\begin{scope}[->,every node/.style={midway,auto},font=\scriptsize]
\foreach \a/\b/\l in {X/Y/f,Y/U/f'} \draw[blue] (\a) -- node[swap]{} (\b); 
\draw[blue] (U) to [bend right] node[swap]{} (TX);
\foreach \a/\b/\l in {TX/TY/T(f),TY/TU/T(f')} \draw (\a) -- node[swap]{} (\b);
\foreach \a/\b/\l in {Y/Z/g,W/TY/g''} \draw[red] (\a) -- node{} (\b);
\draw[red] (Z) to [bend left] node{} (W);
\draw[green] (X) to [bend left] node{} (Z);
\foreach \a/\b/\l in {Z/V/h',V/TX/h''} \draw[green] (\a) -- node{} (\b);
\foreach \a/\b/\l in {U/V/j,V/W/j'} \draw[brown] (\a) -- node{} (\b);
\draw[brown] (W) to [bend left] node{} (TU);
\end{scope}.
\end{tikzpicture} 
\\This yields a distinguished triangle $\mathcal{A}\rightarrow\mathcal{B}\rightarrow\mathcal{C}\rightarrow$ relating $\mathcal{B}$ with the cones of $\mathcal{V}_I(\mu)\la -\delta_\mu\ra\rightarrow A_I(\mu)\la -\delta_\mu\ra$ and $A_I(\mu)\la -\delta_\mu\ra \rightarrow \expcos_I(\mu)$. The object $\mathcal{A}$ is in $\mathrm{D}^I_{\leq \mu}$ and $\mathcal{C}$ is in $\mathrm{D}^I_{< \mu}$ from Lemma \ref{lem:dtri}. Hence, the object $\mathcal{B}$ is in $\mathrm{D}^I_{\leq\mu}$.

For the second part of (3), a similar argument as above works here using different distinguished triangles. The first is $A_I(w_I\mu)\la \delta \ra\rightarrow\mathcal{V}_I(\mu)\la \delta\ra\rightarrow \mathcal{A}'\rightarrow$ with $\mathcal{A}'$ in $\mathrm{D}^I_{\leq \mu}$.
This comes from (a twist of) applying $\upmu{I*}p^*\mathscr{S}_{G/B}$ to the short exact sequence of $B$ representations $0\rightarrow \bbk_{w_I\mu}\rightarrow N_I(\mu)\rightarrow A'\rightarrow 0.$ The second is the distinguished triangle \eqref{eq:alttri} from proof of Lemma \ref{lem:dtri}.
\end{proof}

\subsection{A quasi-exceptional collection} 

For the notion of quasi-exceptional collection, we refer the reader to \cite[2.2]{BEZ2}. See also \cite[Definition 2.4]{A2} for a definition of graded quasi-exceptional collection. For a category $\mathcal{D}$ with a shift the grading functor $\la 1\ra$, we define the notation $\uHom(A, B) = \bigoplus_{n\in\mathbb{Z}}\Hom_\mathcal{D}(A, B\la n\ra)$. The notation $R\uHom$ and $\underline{\Ext}^i$ is defined similarly. 
\begin{prop} \label{prop:cosqex}The objects $(\expcos_I(\lambda), \lambda\in\wts^+_I)$ constitute a graded quasi-exceptional collection with respect to the poset $(\wts^+_I, \leq)$. In particular, they satisfy:
\begin{enumerate}
\item $R\uHom(\expcos_I(\mu), \expcos_I(\lambda)) = 0$ for all $\mu<\lambda$,
\item $\uHom^i(\expcos_I(\mu), \expcos_I(\mu)) = 0$ for $i<0$,
\item $\uHom(\expcos_I(\mu), \expcos_I(\mu)) \cong \bbk$,
\item If $i>0$ and $n\leq 0$, then $\Hom^i (\expcos_I(\mu), \expcos_I(\mu)\la n \ra) = 0$. 
\end{enumerate}

\end{prop} 
\begin{proof} Part (1): We proceed as in the proof of \cite[Proposition 5.6]{A2}. Assume by induction with respect to $<$, $R\uHom(\expcos_I(\gamma), \expcos_I(\lambda)) = 0$ for all $\gamma<\mu <\lambda$. Recall the \textit{star operation} $\star$ from \cite[1.3.9]{BBD}: $B\in A\star C$ if there is a distinguished triangle $A\rightarrow B\rightarrow C\rightarrow$.

By Lemma \ref{lem:vectorbundlehelp} (1), we have $R\uHom(\mathcal{V}_I(\mu), \expcos_I(\lambda)) = 0.$ Hence, Lemma \ref{lem:vectorbundlehelp} (3) and the inductive hypothesis imply there are $m_i\in\mathbb{Z}$ and $\nu_i\in W_I(\mu)$ so that 
\begin{align*}0\in R\uHom&(\expcos_I(\nu_1)\la m_1\ra, \expcos_I(\lambda)) \star\cdots \\ \cdots \star R&\uHom(\expcos_I(\nu_k)\la m_k\ra, \expcos_I(\lambda))\star R\uHom(\expcos_I(\mu)\la m_{k+1}\ra , \expcos_I(\lambda)).\end{align*} By Lemma \ref{lem:stdtostd}, for each $w\in W_I$, we have an isomorphism $\expcos_I(w\mu)\la m\ra\cong \expcos_I(\mu)$ for some $m\in\mathbb{Z}$. Thus, we get \begin{align*}0\in R\uHom&(\expcos_I(\mu)\la m'_1\ra, \expcos_I(\lambda)) \star\cdots \\ \cdots \star R&\uHom(\expcos_I(\mu)\la m'_k\ra, \expcos_I(\lambda))\star R\uHom(\expcos_I(\mu)\la m'_{k+1}\ra , \expcos_I(\lambda)).\end{align*} By \cite[Lemma 2.2 (1)]{A2}, we have $R\uHom(\expcos_I(\mu), \expcos_I(\lambda)) = 0$  as desired. 


Parts (2)-(4): Let $V=R\uHom(\expcos_I(\mu), \expcos_I(\mu))$. Lemma \ref{lem:vectorbundlehelp} (2) gives an isomorphism $R\uHom(\mathcal{V}_I(\mu), \expcos_I(\mu)\la\delta_\mu\ra)\cong \bbk$, so a similar argument as above along with part (1) imply there are $n_1, \ldots, n_k >0$ so that $\bbk \in V\star V\la n_1\ra \star\cdots\star V\la n_k\ra$. Thus, \cite[Lemma 2.2 (2)]{A2} yields parts (2), (3), and (4). That is, for $i<0$, $\co^i(V)=\uHom^i(\expcos_I(\mu), \expcos_I(\mu)) = 0$. We have $\co^0(V)=\uHom(\expcos_I(\mu), \expcos_I(\mu)) \cong \bbk$ so in particular, $\Hom(\expcos_I(\mu), \expcos_I(\mu)\la m\ra) = 0$ for $m\neq 0$. For $i>0$, $\co^i(V) = \uHom^i(\expcos_I(\mu), \expcos_I(\mu))$ is concentrated in positive degrees, so for $m\leq 0$, we have $\Hom^i(\expcos_I(\mu), \expcos_I(\mu)\la m\ra) = 0$. \end{proof}

\begin{remark} A similar strategy as in this section using the partial order $\wtles_I$ on $\wts^+_I$ proves the collection $(A_I(\lambda), \lambda\in\wts^+_I)$ is also graded quasi-exceptional. However, this partially ordered set cannot be refined to be isomorphic to $(\mathbb{Z}_{\geq}, \leq)$ (except when $I = \Sigma$), so the results of \cite{BEZ2} will not apply to give an associated t-structure.  
\end{remark}

By the same techniques, the proper standard objects also form a quasi-exceptional collection with respect to the opposite order on $\wts^+_I$. 
\begin{prop} \label{prop:stdqex}The objects $(\expstd_I(\lambda), \lambda\in\wts^+_I)$  satisfy:
\begin{enumerate}
\item $R\uHom(\expstd_I(\lambda), \expstd_I(\mu)) = 0$ for all $\mu<\lambda$,
\item $\uHom^i(\expstd_I(\mu), \expstd_I(\mu)) = 0$ for $i<0$,
\item $\uHom(\expstd_I(\mu), \expstd_I(\mu)) \cong \bbk$,
\item If $i>0$ and $n\leq 0$, then $\Hom^i (\expstd_I(\mu), \expstd_I(\mu)\la n \ra) = 0$.
\end{enumerate}
\end{prop}   

\begin{cor} \label{cor:mut}
Let $\lambda\in\wts^+_I$. 
\begin{enumerate}
\item $\mathrm{D}^I_{< \lambda}$ coincides with the triangulated category generated by the collection $(\expcos_I(\gamma)\la m\ra,\gamma<\lambda \textup{ with } \gamma\in\wts^+_I, m\in\mathbb{Z})$.
\item $\mathrm{D}^I_{< \lambda}$ coincides with the triangulated category generated by the collection $(\expstd_I(\gamma)\la m\ra, \gamma<\lambda \textup{ with } \gamma\in\wts^+_I, m\in\mathbb{Z})$.
\item For all $\mathcal{G}\in  \mathrm{D}^I_{< \lambda}$, we have $\Hom(\mathcal{G}, \expcos_I(\lambda)) = \Hom(\expstd_I(\lambda), \mathcal{G}) = 0$.
\end{enumerate}
\end{cor}

\begin{proof} Part (1) can be proven by induction on $(\wts^+_I, \leq)$ using the first distinguished triangle in Lemma \ref{lem:dtri}. Part (2) is similar. Now, part (3) follows from the first two parts along with part (1) of Propositions \ref{prop:cosqex} and \ref{prop:stdqex}.
\end{proof}

\begin{remark} The distinguished triangles in Lemma \ref{lem:dtri} along with Corollary \ref{cor:mut} justify referring to the collection $(\expcos_I(\lambda), \lambda\in\wts^+_I)$ as the $\leq$-{\em{mutation}} of the collection $(A_I(\lambda), \lambda\in\wts^+_I)$, although we were not able to perform mutation directly. 
\end{remark}

For each $\mu\in\wts^+_I$, we get a morphism \begin{equation}\label{eq:pstdpcosmap}\expstd_I(\mu)\rightarrow \expcos_I(\mu)\end{equation} by composing the maps $\expstd_I(\mu)\rightarrow A_I(\mu)\la -\delta_\mu\ra$ and $A_I(\mu)\la -\delta_\mu\ra\rightarrow\expcos_I(\mu)$ from the triangles in Lemma \ref{lem:dtri}.
Next, we must confirm that our two quasi-exceptional sets above are dual to each other in the sense of \cite[2.2]{BEZ2}. See also \cite[Definition 2.6]{A2}.
\begin{prop}\label{prop:dual} The collection of objects $(\expstd_I(\lambda), \lambda\in\wts^+_I)$ satisfy
\begin{enumerate}
\item $\expstd_I(\lambda)\cong \expcos_I(\lambda) \,\mathrm{ mod }\, \mathrm{D}^I_{<\lambda}$
\item If $\lambda>\mu$, then $R\uHom(\expstd_I(\lambda), \expcos_I(\mu)) = 0$
\end{enumerate}
That is, the quasi-exceptional collection $(\expcos_I(\lambda), \lambda\in\wts^+_I)$ is dualizable with dual collection $(\expstd_I(\lambda), \lambda\in\wts^+_I)$.
\end{prop}
\begin{proof} Part (1) is an immediate consequence of Lemma \ref{lem:dtri}. For part (2), the first distinguished triangle in Lemma \ref{lem:dtri} and that $\mu<\lambda$ imply that $\expcos_I(\mu)\in\mathrm{D}^I_{<\lambda}$. Hence, the required vanishing follows from Corollary \ref{cor:mut} part (3).
\end{proof}

Now we come to the main theorem of this section which defines the {\it exotic} t-structure on $\der\coh^{G \times \mathbb{G}_m}(\spr^I)$. In the statement, we use the notation $\la S \ra$ for a set of objects $S$ to mean the full subcategory generated by extensions of objects in $S$.
\begin{thm}\label{thm:exotict-str}
There is a $t$-structure $(^{E}{\mathcal{D}}^{\leq 0}, ^{E}{\!\mathcal{D}}^{\geq 0})$ on $\mathcal{D}:=\der\coh^{G \times \mathbb{G}_m}(\spr^I)$ 
 given by 
 \[^{E}{\mathcal{D}}^{\leq 0} = \la\{ \expstd_I(\lambda)\la m\ra[d] \,|\, d\geq0, \lambda\in\wts^+_I, m\in\mathbb{Z}\} \ra\]
\[^{E}{\mathcal{D}}^{\geq 0} = \la\{ \expcos_I(\lambda)\la m\ra[d] \,|\, d\leq0, \lambda\in\wts^+_I, m\in\mathbb{Z}\}\ra.\]
Moreover, $\upmu{I*}$ is t-exact with respect to these exotic t-structures, and the heart $\mathrm{ExCoh}^{G \times \mathbb{G}_m}(\spr^I) \subset \der\coh^{G \times \mathbb{G}_m}(\spr^I)$
is stable under $\langle 1 \rangle$ and contains $\expcos_I(\lambda)$ and $\expstd_I(\lambda)$ for all $\lambda\in\wts^+_I$.
\end{thm}
\begin{proof}Propositions \ref{prop:cosqex} and Corollary \ref{prop:dual} yield that the objects $(\expcos_I(\lambda), \wts^+_I)$ form a dualizable graded quasi-exceptional collection with dual collection $(\expstd_I(\lambda), \wts^+_I)$. Moreover, Corollary \ref{lem:cosgen} says that they also generate our category, and our set $(\wts^+_I, <)$ is well-ordered. Thus, \cite[Proposition 1]{BEZ2} proves we have a t-structure. (Technically, \cite[Proposition 1]{BEZ2} is considering the ungraded case, but the same argument works for the graded case as well, see \cite[Proposition 4]{BEZ}.) It's clear the heart is stable under $\la 1\ra$, and that $\upmu{I*}$ is t-exact by definition of the two t-structures and our definition of the proper (co) standards. Note that \cite[Proposition 1]{BEZ2} does not guarantee that proper standards and proper costandards are in the heart. However, it is known that $\excos(\lambda)$ and $\exstd(\lambda)$ (for all $\lambda\in\wts$) are in $\mathrm{ExCoh}^{G \times \mathbb{G}_m}(\spr)$ by \cite[Corollary 3.10]{MR}, so the same is true here by exactness of $\upmu{I*}$. \end{proof}

\begin{remark} \begin{enumerate}
\item Exactness of $\upmu{I*}$ along with \cite[Proposition 3.12, (1)]{MR} imply the generalized Andersen--Jantzen sheaf $A_I(\lambda)$ is exotic for all $\lambda\in\wts$. Similarly, $\mathcal{V}_I(\lambda)$ for $\lambda\in\wts^+_I$ is exotic since $\mathcal{V}(\lambda)$ has a filtration by line bundles.
\item It is easy to see that the functor $\mu_{I*}: \der\coh^{G\times\gm}(\spr^I)\rightarrow \der\coh^{G\times\gm}(\nilp)$ is also exact taking exotic sheaves to perverse coherent sheaves. 
\item Lemma \ref{lem:inducedaj} proves that the ``induction" functor $\der\coh^{G\times\gm}(G\times^{P_I}\nilp_L)\rightarrow \der\coh^{G\times\gm}(\spr^I)$ is exact as well taking perverse coherent sheaves to exotic sheaves since the generalized Andersen--Jantzen sheaves are exotic.
\end{enumerate}
\end{remark}

\section{Study of the heart}\label{Sec: heart}
The t-structure in \cite[Proposition 1]{BEZ2} arising from a quasi-exceptional collection is obtained by gluing or recollement. For each $\lambda\in \wts^+_I$, let $i_\lambda: \mathrm{D}^I_{<\lambda}\rightarrow\mathrm{D}^I_{\leq\lambda}$ and $\Pi_\lambda: \mathrm{D}^I_{\leq\lambda}\rightarrow \mathrm{D}^I_{\leq\lambda}/\mathrm{D}^I_{<\lambda}$ denote the natural inclusion and quotient functors. Then \cite[Lemma 4]{BEZ2} guarantees existence of left and right adjoints to both $i_\lambda$ and $\Pi_\lambda$ which we denote by $i^L_\lambda$, $i^R_\lambda$, $\Pi^L_\lambda$, and $\Pi^R_\lambda$.
That is, for each $\lambda\in \dom_I$, we have a recollement diagram 

\[
\begin{tikzpicture}
\node (a) at (0,0) {$\mathrm{D}^I_{<\lambda}$};
\node (b) at (4,0) {$\mathrm{D}^I_{\leq\lambda}$};
\node (c) at (8,0) {$\mathrm{D}^I_{\leq\lambda}/\mathrm{D}^I_{<\lambda}$};
\draw[-latex] (a) -- (b);
\draw[-latex, above] (a) to node {$i_\lambda$} (b);
\draw[-latex, above] (b) to node {$\Pi_\lambda$} (c);
\draw[-latex,bend right=30, above] (b)  to node {$i_\lambda^L$} (a);
\draw[-latex,bend left=30, above] (b) to node {$i_\lambda^R$} (a);
\draw[-latex,bend right=30, above]  (c) to node {$\Pi_\lambda^L$} (b);
\draw[-latex,bend left=30, above]  (c) to node {$\Pi_\lambda^R$} (b);
\end{tikzpicture}
\]
Each of the above categories has induced compatible t-structures such that $i_\lambda, \Pi_\lambda$ are exact, $i^R_\lambda, \Pi_\lambda^R$ are left exact, and $i^L_\lambda, \Pi_\lambda^L$ are right exact, see \cite[Section 1.4, Proposition 1.4.16]{BBD}.
Now, Lemma \ref{lem:dtri} and Corollary \ref{cor:mut} part (3) imply \[\expcos_I(\lambda)\cong \Pi_\lambda^R\Pi_\lambda(A_I(\lambda)\la-\delta_\lambda\ra) \hspace{.5cm}\textup{ and }\hspace{.5cm}\expstd_I(\lambda)\cong\Pi_\lambda^L\Pi_\lambda (A_I(\lambda)\la -\delta_\lambda\ra).\] For each $\lambda\in \wts^+_I$, we also define objects 
\begin{equation} \label{eq:(co)std}\excos_I(\lambda):=  \Pi_\lambda^R\Pi_\lambda(\mathcal{V}_I(\lambda)\la-2\delta^I_{w_I\lambda}-\delta_\lambda\ra)\hspace{.5cm}\textup{ and }\hspace{.5cm}\exstd_I(\lambda):=\Pi_\lambda^L\Pi_\lambda (\mathcal{V}_I(\lambda)\la -\delta_\lambda\ra),\end{equation} and refer to them as costandard and standard objects respectively. Recall also the natural map from \eqref{eq:pstdpcosmap} $\expstd_I(\lambda)\rightarrow \expcos_I(\lambda)$; denote the image by $\exirr_I(\lambda)$. Then \cite[Proposition 2]{BEZ2} guarantees that $\exirr_I(\lambda)$ is irreducible, and any irreducible in $\mathrm{ExCoh}^{G \times \mathbb{G}_m}(\spr^I)$ is isomorphic to  $\exirr_I(\lambda)\la m\ra$ for some $\lambda\in\wts^+_I, m\in\mathbb{Z}$. 

The following proposition shows $\upmu{I*}:\mathrm{ExCoh}^{G \times \mathbb{G}_m}(\spr)\rightarrow\mathrm{ExCoh}^{G \times \mathbb{G}_m}(\spr^I)$ factors through a Serre quotient. 

\begin{prop} The exact functor $\upmu{I*}:\mathrm{ExCoh}^{G \times \mathbb{G}_m}(\spr)\rightarrow\mathrm{ExCoh}^{G \times \mathbb{G}_m}(\spr^I)$ satisfies $\upmu{I*}(\exirr(\lambda)) = 0$ in case $\lambda\not\in-\wts^+_I$.  \end{prop}
\begin{proof} By definition, the object $\exirr(\lambda)$ is image of a nonzero map $h: \exstd(\lambda)\rightarrow \excos(\lambda)$. Hence, using Lemma \ref{lem:stdtostd}, $\upmu{I*}(\exirr(\lambda))$ is image of the map \[\upmu{I*}(h): \expstd(\dm_I(\lambda))\la2\delta^I_{\mathrm{ant}_I(\lambda)}-\delta^I_\lambda\ra\rightarrow \expcos(\dm_I(\lambda))\la\delta^I_\lambda\ra,\] where $\mathrm{ant}_I(\lambda)$ denotes the unique element in $-\wts^+_I\cap W_I(\lambda)$. Properties of dualizable quasi-exceptional collections (namely Proposition \ref{prop:stdqex} part (3) together with Proposition \ref{prop:dual} part (1)) guarantee that $\Hom(\expstd(\dm_I(\lambda))\la m\ra, \expcos(\dm_I(\lambda))\la n\ra)=0$ unless $m=n$. In other words, $\upmu{I*}(\exirr(\lambda)) = 0$ unless $\lambda\in-\wts^+_I$.
\end{proof}

\subsection{Properly Stratified Categories}
Suppose $\mathcal{A}$ is a $\bbk$-linear abelian category. Assume $\mathcal{A}$ is \textit{graded}, that is, there is an automorphism $\la1\ra:\mathcal{A}\rightarrow\mathcal{A}$ which acts as a \textit{shift the grading} functor. In particular, $\la 1\ra$ permutes $\mathrm{Irr}(\mathcal{A})$, the set of (isomorphism classes of) irreducible objects in $\mathcal{A}$. Let $\Omega = \mathrm{Irr}(\mathcal{A})/\mathbb{Z}$, and for each $\gamma\in\Omega$, choose a representative simple object $L_\gamma$ in $\mathcal{A}$. Then any simple object in $\mathcal{A}$ is isomorphic to $L_\gamma\la n\ra$ for some $\gamma\in\Omega$ and $n\in\mathbb{Z}$.

Assume $\Omega$ is partially ordered with respect to $\leq$, and assume for any $\gamma\in\Omega$, the set $\{\xi\in\Omega | \xi\leq\gamma\}$ is finite. Given a finite order ideal $\Gamma\subset \Omega$, we let $\mathcal{A}_\Gamma\subset\mathcal{A}$ denote the Serre subcategory of $\mathcal{A}$ generated by the collection of simple objects $\{L_\gamma\la n\ra | \gamma\in\Gamma, n\in\mathbb{Z}\}$. We simplify notation in the special case $\mathcal{A}_{\leq \gamma} := \mathcal{A}_{\{\xi\in\Omega | \xi\leq \gamma\}}$. The category $\mathcal{A}_{<\gamma}$ is defined similarly. We recall the notion of a graded properly stratified category. See \cite{d, fm} for details on the corresponding notion for algebras.

\begin{defn}\label{def:propstr} Suppose $\mathcal{A}$, $\Omega$, and $\leq$ are as above. We say $(\mathcal{A}, \Omega, \leq)$ is \textit{graded properly stratified} if for each $\gamma\in\Omega$, the following conditions hold:
\begin{enumerate}
\item We have $\End(L_\gamma)\cong\bbk$.
\item There is an object $\bar{\Delta}_\gamma$ and a surjective morphism $\phi_\gamma: \bar{\Delta}_\gamma\rightarrow L_\gamma$ such that \[\ker(\phi_\gamma)\in\mathcal{A}_{<\gamma} \textup{ and } \uHom(\bar{\Delta}_\gamma, L_\xi) = \underline{\Ext}^1(\bar{\Delta}_\gamma, L_\xi)=0 \textup{ if } \xi\not\geq\gamma. \]
\item There is an object $\bar{\nabla}_\gamma$ and an injective morphism $\psi_\gamma: L_\gamma\rightarrow \bar{\Delta}_\gamma$ such that \[\mathrm{coker}(\psi_\gamma)\in\mathcal{A}_{<\gamma} \textup{ and } \uHom(L_\xi, \bar{\nabla}_\gamma) = \underline{\Ext}^1(L_\xi, \bar{\nabla}_\gamma)=0 \textup{ if } \xi\not\geq\gamma. \]
\item In $\mathcal{A}_{\leq \gamma}$, $L_\gamma$ admits a projective cover $\Delta_\gamma\rightarrow L_\gamma$. Moreover, $\Delta_\gamma$ admits a filtration whose subquotients are of the form $\bar{\Delta}_\gamma\la n\ra$ for various $n\in\mathbb{Z}$.
\item In $\mathcal{A}_{\leq \gamma}$, $L_\gamma$ admits an injective envelope $L_\gamma\rightarrow\nabla_\gamma $. Moreover, $\nabla_\gamma$ admits a filtration whose subquotients are of the form $\bar{\nabla}_\gamma\la n\ra$ for various $n\in\mathbb{Z}$.
\item We have $\underline{\Ext}^2(\Delta_\gamma, \bar{\nabla}_\xi) = \underline{\Ext}^2(\bar{\Delta}_\gamma, \nabla_\xi) = 0$ for all $\gamma, \xi\in\Omega$.
\end{enumerate} We call an object in $\mathcal{A}$ \textit{standard}, (resp. \textit{costandard, proper standard, proper costandard}) if it is isomorphic to some $\Delta_\gamma\la n\ra$ (resp. $\nabla_\gamma\la n\ra, \bar{\Delta}_\gamma\la n\ra, \bar{\nabla}_\gamma\la n\ra$). 

\end{defn}

If $\mathcal{A}$ also satisfies (7) $\Delta_\gamma\cong \bar{\Delta}_\gamma$ and $\nabla_\gamma\cong \bar{\nabla}_\gamma$ for all $\gamma\in\Omega$, then $\mathcal{A}$ is called a \textit{graded highest weight category} (see \cite[2.2]{fm}). 

\begin{thm}\label{thm:propstr} The category $\mathrm{ExCoh}^{G \times \mathbb{G}_m}(\spr^I)$ is graded properly stratified with respect to the ordered set $(\wts^+_I, \leq)$. The standard, costandard, proper standard, and proper costandard objects are given by $\exstd(\lambda), \excos_I(\lambda), \expstd_I(\lambda),$ and $\expcos_I(\lambda)$ respectively for $\lambda\in\wts^+_I$. 
\end{thm}
\begin{proof} Let $\mathrm{E}^I$ denote $\mathrm{ExCoh}^{G \times \mathbb{G}_m}(\spr^I)$ within this proof. We must verify the properties of Definition \ref{def:propstr}. Property (1) is immediate. Property (2) and (3) follow since the t-structure is defined using a quasi-exceptional collection. See \cite[Proposition 2]{BEZ2}; note that {\it quasi-hereditary} is given a non-standard meaning in that reference. 

To verify (4), we will check that the object $\exstd_I(\lambda)$ as defined in \eqref{eq:(co)std} satisfies the required conditions. If $A\in\mathrm{E}^I_{<\lambda}$, then we have $\Pi_\lambda(A) = 0$, so $\Hom_{\mathrm{D}^I}(\exstd_I(\lambda), A[n]) = 0$ for all $n$. This implies \begin{equation}\label{eq:ext0}\Ext^n_{\mathrm{E}^I}(\exstd_I(\lambda), A)=\Ext^n_{\mathrm{E}^I_{\leq\lambda}}(\exstd_I(\lambda), A)=0.\end{equation} for all $n\geq 0$ as well using \cite[Remarque 3.1.17]{BBD}. 


The quotient category $\mathrm{D}^I_{\leq\lambda}/\mathrm{D}^I_{<\lambda}$ admits a t-structure whose heart has exactly one irreducible object (up to grading shift) $\Pi_{\lambda}(\exirr_I(\lambda))$ which is isomorphic to $\Pi_{\lambda}(\expstd_I(\lambda))$ by property (2). The object $\mathcal{V}_I(\lambda)\la -\delta_\lambda\ra$ is an extension (in the sense of triangulated categories) of $A_I(\lambda)\la -\delta_\lambda\ra$ by other generalized Andersen--Jantzen sheaves. Hence, the object $\Pi_{\lambda}(\mathcal{V}_I(\lambda)\la -\delta_\lambda\ra)\cong\Pi_{\lambda}(\exstd_I(\lambda))$ has a filtration by $\Pi_{\lambda}(\expstd_I(\lambda))\la m\ra$ for various $m$. Moreover, by applying $\Pi_\lambda^L\circ\Pi_\lambda$ to the nonzero map $\mathcal{V}_I(\lambda)\la -\delta_\lambda\ra\rightarrow A_I(\lambda)\la-\delta_\lambda\ra$, we have nonzero map $\Pi_{\lambda}(\exstd_I(\lambda))\rightarrow \Pi_{\lambda}(\exirr_I(\lambda))$ which is necessarily surjective. Since $\Pi^L$ is right exact, we also get an epimorphism $\exstd_I(\lambda)\rightarrow \exirr_I(\lambda)$.

Consider the short exact sequence \[0\rightarrow K\rightarrow \expstd_I(\lambda)\rightarrow\exirr_I(\lambda)\rightarrow 0\] with $K \in \mathrm{E}^I_{<\lambda}$. Apply $\Hom(\exstd_I(\lambda), -)$ to get a long exact sequence. \eqref{eq:ext0} implies $\Hom(\exstd_I(\lambda), K[n])=0$ for all $n$. A straight-forward adjunction argument also shows $\Hom(\exstd_I(\lambda), \expstd_I(\lambda)[n]) \cong \Hom(\mathcal{V}_I(\lambda), \expcos_I(\lambda)\la\delta_{\lambda}\ra[n])$ which vanishes for $n\neq0$ by Lemma \ref{lem:vectorbundlehelp}, part (2). Hence, $\Hom(\exstd_I(\lambda), \exirr_I(\lambda)[n]) = 0$ for all $n$ as well. This verifies that $\exstd_I(\lambda)$ is projective cover of $\exirr_I(\lambda)$. 

Now we turn to property (5). The inclusion of $B$-modules $ \bbk_{w_I\lambda}\hookrightarrow N_I(\lambda)$ gives rise in the usual way to the nonzero morphism $A_I(w_I\lambda)\la-2\delta^I_{w_I\lambda}-\delta_\lambda\ra\rightarrow \mathcal{V}_I(\lambda)\la-2\delta^I_{w_I\lambda}-\delta_\lambda\ra$. To get the morphism $\exirr_I(\lambda)\rightarrow\excos_I(\lambda)$, we apply $\Pi^R_\lambda\circ\Pi_\lambda$ to the above map noting that there is an isomorphism $\expcos_I(\lambda)\cong \Pi^R_\lambda\Pi_\lambda(A_I(w_I\lambda)\la-2\delta^I_{w_I\lambda}-\delta_\lambda\ra)$ by Lemma \ref{lem:waj}. The rest of the argument follows in a similar manner to (4).

Property (6) is easily seen to hold using properties of the quotient functors along with \cite[Remarque 3.1.17]{BBD}.
\end{proof}

\end{document}